\documentclass{amsart}
\usepackage{amsmath,amssymb,amsthm,hyperref}
\usepackage{amsfonts,amsmath,mathrsfs,amssymb,graphicx,multirow,amsthm,latexsym}
\usepackage[mathscr]{euscript}
\usepackage[all]{xy}
\usepackage{enumerate}
\usepackage{multicol}
\usepackage{cite}

\allowdisplaybreaks

\newcommand{\ZZ}{\mathbb Z}

\newcommand{\PP}{\mathbb P}
\newcommand{\Xdif}[1]{\partial^{\mathcal{X}}\hspace{-3pt}\left({#1}\right)}
\newcommand{\Ydif}[1]{\partial^{\mathcal{Y}}\hspace{-3pt}\left({#1}\right)}
\newcommand{\ind}[1]{\mathbf{1}_{\left[{#1}\right]}}
\newcommand{\mprod}[1][]{*_{#1}}
\newcommand{\amput}[2][\mathcal{X}]{{#1} \mprod[R] {#2}}

\newcommand{\Mdif}[1]{\partial^{\mX * \mY}\hspace{-3pt}\left({#1}\right)}

\newcommand{\vv}[1]{\underline{{#1}}}
\newcommand{\Cone}[1]{\operatorname{Cone}\left({#1}\right)}
\newcommand{\Hm}[2]{\operatorname{H}_{#1}\left({#2}\right)}

\newcommand{\mI}{\mathcal{I}}
\newcommand{\mJ}{\mathcal{J}}
\newcommand{\mX}{\mathcal{X}}
\newcommand{\mY}{\mathcal{Y}}
\newcommand{\maxM}{\mathfrak{M}}
\DeclareMathOperator{\IM}{Im}
\newcommand{\cS}{\mathscr{S}}
\newcommand{\cT}{\mathscr{T}}
\DeclareMathOperator{\rank}{rank}

\newcommand{\Tor}[4][R]{\operatorname{Tor}^{#1}_{#2}\left(#3,#4\right)}

\theoremstyle{plain}
\newtheorem{theorem}{Theorem}
\newtheorem{prop}[theorem]{Proposition}
\newtheorem{cor}[theorem]{Corollary}
\newtheorem{lemma}[theorem]{Lemma}

\newtheorem{construct}[theorem]{Construction}
\newtheorem{note}[theorem]{Note}

\theoremstyle{definition}

\newtheorem{notation}[theorem]{Notation}

\newtheorem{example}[theorem]{Example}

\numberwithin{theorem}{section}
\numberwithin{equation}{section}
\numberwithin{equation}{theorem}

\begin{document}

\title{Minimal Free Resolutions of Fiber Products}
\author[H. Geller]{Hugh Geller$^1$}
\address{$^1$School of Mathematical and Statistical Sciences\\
Clemson University\\ Clemson, SC 29634 USA}
\email{hgeller@g.clemson.edu}
\urladdr{https://hughgeller.com}

\thanks{This work was partially funded by a Clemson University Doctoral Dissertation Completion Grant.}

\begin{abstract}
We construct free resolutions for quotient rings $R/\langle \mathcal{I}', \mathcal{I}\mathcal{J}, \mathcal{J}'\rangle$, give conditions for the quotient to be realized as a fiber product, and give criteria for the construction to be minimal. We then specialize this result to fiber products over a field $k$ and recover explicit formulas for Betti numbers, graded Betti numbers, and Poincar\'{e} series.
\end{abstract}
\date{\today}
\maketitle


\section{Introduction}

Throughout the paper, let $(S, \maxM_S,k)$ and $(T, \maxM_T,k)$ both be commutative, local (or standard graded), notherian rings with $S \cong A/I'$ and $T \cong B/J'$ where $A$ and $B$ are regular local rings (or polynomial rings over $k$). Both $S$ and $T$ come equipped with a natural surjection $S \xrightarrow{\pi_S} k \xleftarrow{\pi_T} T$ which are used to build the fiber product of $S$ and $T$ over $k$ given by $S \times_k T := \{(s,t) \in S \times T: \pi_S(s) = \pi_T(t)\}$. Much research has been conducted comparing and contrasting the homological properties of $S$ and $T$ with those of $S \times_k T$, e.g., the Cohen-Macaulay, Gorenstein, Golod, finite representation type, and Arf properties (see \cite{celikbas2020weakly, MR2580452, freitas2019vanishing, MR647683, MR2488551, MR2711725, MR3691985, MR3862678, nasseh2018gorenstein,  MR3988200, MR3754407}).




We continue this line of inquiry by explicitly constructing a minimal free resolution of the fiber product $F = S \times_k T$ over an appropriate regular local ring or polynomial ring; see Theorem \ref{thm:minimal} and Corollary \ref{cor:fiber}. In particular, this construction yields the following formulas for Poincar\'e series; see Corollary \ref{cor:poin}.


\begin{theorem}\label{thm:main}
Set $\vv{x} = x_1, \ldots, x_m$ and $\vv{y} = y_1, \ldots, y_n$. Suppose $\mI' \subseteq \langle \vv{x} \rangle^2$ and $\mJ' \subseteq \langle \vv{y} \rangle^2.$ Consider either of the two following cases with $k$ a field;
\begin{enumerate}
	\item $A = k[\vv{x}]$ and $B = k[\vv{y}]$ with $R = k[\vv{x},\vv{y}]$;
	\item $A = k[[\vv{x}]]$ and $B = k[[\vv{y}]]$ with $R = k[[\vv{x},\vv{y}]]$.
\end{enumerate}
Set $S = A/I'$ and $T = B / J'$ and consider the fiber product $F:= S \times_k T$. Then one has the following formulas for Poincar\'{e} series: \[\frac{P_{F}^{R}(t) - P_{R/(\mI' + \maxM_B )}^{R}(t) - P_{R/(\maxM_A + \mJ')}^{R}(t) + P_{R/\maxM_R}^{R}(t)}{P_{\langle \vv{x} \vv{y} \rangle}^{R}(t)} = t(1+t)\] as well as \[\frac{P_{F}^{R}(t) - (1+t)^n P_S^{A}(t) - (1+t)^m P_T^{B}(t) + (1 + t)^{m + n}}{\left( (1 + t)^m - 1 \right) \left( (1 + t)^n - 1 \right)} = \frac{t + 1}{t}. \] 
\end{theorem}



Section 2 of this paper documents background material for use in Sections 3 and~4. In addition, this section contains a few results about indicator functions (inspired by the use of measure theory within probability and stochastics) which we use to reduce significantly the number of cases needed for our results and proofs.

Section 3 is devoted to Construction \ref{cons} that builds a minimal resolution over a polynomial ring for a quotient of that polynomial ring by certain products of ideals. One can then recover the results of \cite{MR2262383} on edge ideals of complete bipartite graphs by specializing Theorem \ref{thm:resn} using $\mI = \langle \vv{x}\rangle$ and $\mJ = \langle \vv{y} \rangle$.

Section 4 utilizes (hard) truncations, tensor products, and shifts in order to manipulate the minimal resolutions of $S$ and $T$ over distinct polynomial rings into resolutions over our desired ring. Constuction \ref{con:lift} shows how to build specific chain maps from these new resolutions to the one constructed in Theorem \ref{thm:genres} so that the associated mapping cones provide resolutions of the desired quotients.

Section 5 is dedicated to the criteria for and consequences of the construction in Theorem \ref{thm:genres} being minimal. Corollary \ref{cor:fiber} specifically identifies the application to fiber products. The section concludes by giving formulas for the Betti numbers (Corollary \ref{cor:betti}) and Poincar\'{e} series (Corollary \ref{cor:poin}) for the minimal construction.

\section{Notation and Background}

Throughout the paper, let $(R,\maxM_R)$ be a local or standard graded rings with flat ring. If $\cS$ is a chain complex of finite-rank free $R$-modules such that $\cS_i = 0$ for all $i < 0$, then its {\bf generating function} is \[\PP_{\cS}^{R}(t) = \sum_{n\geq 0} \rank_{R} \cS_n t^n.\] If $\cS$ is acyclic and minimal, then the {\bf Poincar\'{e} series} for $\Hm{0}{\cS}$ is $P_{\Hm{0}{\cS}}(t) = \PP_{\cS}(t)$.

We denote the $\ell$th {\bf suspension} (or {\bf shift}) of $\cS$ as $\Sigma^{\ell} \cS$. In the case $\ell = 1$, we set $\Sigma \cS := \Sigma^1 \cS$. We write $\cS_{\geq p}$ for the (hard) truncation of $\cS$ in degrees greater than or equal to $p$. This complex is given by
\begin{align*}
\left(\cS_{\geq p}\right)_m = \begin{cases} \cS_q  & q\geq p \\ 0 & q < p \end{cases} & & \hbox{ and } & & \partial_q^{\cS_{\geq p}} = \begin{cases} \partial_q^{\cS} & q > p \\ 0 & q\leq p \end{cases}
\end{align*}

Two $R$-modules $M$ and $N$ are said to be {\bf Tor-independent} if $\Tor{i}{M}{N} = 0$ for all $i \geq 1$. If $\mI, \mJ \subseteq R$ are ideals such that $R/\mI$ and $R/\mJ$ are $\operatorname{Tor}$-independent, then we say $\mI$ and $\mJ$ are $\operatorname{Tor}$-independent.


\begin{prop}\label{prop:fiber}
Consider the ideals $\mI' \subseteq \mI$ and $\mJ' \subseteq \mJ$, all in $R$, such that $\mI \cap \mJ = \mI\mJ$, e.g., such that $\mI$ and $\mJ$ are Tor-independent. Set $W = R / (\mI + \mJ)$. We then have \[\frac{R}{\langle \mI',\mI\mJ, \mJ' \rangle} \cong \frac{R}{\mI' + \mJ} \times_W \frac{R}{\mI + \mJ'}.\]
\end{prop}

\begin{proof}
Set $F = \frac{R}{\mI' + \mJ} \times_W \frac{R}{\mI + \mJ'} = S \times_W T$ and note that $\mI' \subseteq \mI$ and $\mJ' \subseteq \mJ$ gives us natural surjections $\pi_S: S \to W \leftarrow T :\pi_T$. Moreover, fiber products come with a universal mapping property making the following diagram commute
\[
\xymatrix{
R \ar@{-->}[rd]^-{\exists !\mu} \ar@/^1pc/[rrd]^{p_2} \ar@/_1pc/[ddr]_{p_1} & & & \\
 & F \ar[d] \ar[r] & T \ar[d]^-{\pi_T} \\
 & S \ar[r]^-{\pi_S} & W
}
\] where $p_1$ and $p_2$ are the natural surjections from $R$ to $S$ and $T$, respectively. It is straightforward to check that $\mu$ is surjective and yields the following isomorphism.
\[ F \cong \frac{R}{\ker \mu} = \frac{R}{\ker p_1 \cap \ker p_2} = \frac{R}{(\mI' + \mJ) \cap (\mI + \mJ')}.\] From here we observe that \[\mI' + \mI\mJ + \mJ' \subseteq (\mI' + \mJ) \cap (\mI + \mJ').\]

To obtain the other containment, we consider $\alpha \in (\mI' + \mJ) \cap (\mI + \mJ')$. We can then write $\alpha = \alpha_{\mI'} + \alpha_{\mJ} = \alpha_{\mI} + \alpha_{\mJ'}$ where the subscript denotes which ideal each element comes from, i.e., $\alpha_{\mI'} \in \mI'$. Set $\beta = \alpha_{\mI'} - \alpha_{\mI}$ and note that $\mI' \subseteq \mI$ and $\mJ' \subseteq \mJ$ means we have \[\beta = \alpha_{\mI'} - \alpha_{\mI} = \alpha_{\mJ'} - \alpha_{\mJ} \in \mI \cap \mJ = \mI\mJ.\] From here we can conclude that \[\alpha = \alpha_{\mI} + \alpha_{\mJ'} = \alpha_{\mI'} - \beta + \alpha_{\mJ'} \in \mI' + \mI\mJ + \mJ',\] which then gives the desired containment and isomorphism of rings.
\end{proof}

\begin{example}\label{ex:fiber}
Consider rings $S \cong k[[\vv{x}]]/I'$ and $T \cong k[[\vv{y}]]/J'$. Set $R = k[[\vv{x},\vv{y}]]$ and consider the ideals of $\mI' = I'R$ and $\mJ' = J'R$. We then have $S \cong R/(\mI' + \langle \vv{y} \rangle)$ and $T \cong R/(\langle \vv{x} \rangle + \mJ')$. Combining these isomorphisms with Proposition \ref{prop:fiber} yields the following.\[S \times_k T \cong \frac{R}{\mI' + \langle \vv{y} \rangle} \times_{\frac{R}{\langle \vv{x} \rangle + \langle \vv{y} \rangle}} \frac{R}{\langle \vv{x} \rangle + \mJ'} \cong \frac{R}{\langle \mI', \vv{x} \vv{y}, \mJ'\rangle}.\] We further note that this isomorphism holds if the power series rings are replaced with polynomial rings.
\end{example}


Throughout this paper, we utilize indicator functions $\ind{*}$. These functions return 1 if the input makes the statement $*$ true and 0 if false. In particular, we make use of the following properties.

\begin{lemma}\label{lem:ind} 
Let $W_1$ and $W_2$ be statements that can be evaluated as true or false, and let $W_1^c$ represent ``not $W_1$". Then the following hold.
\begin{enumerate}[\rm(a)]
\item \label{item:AnotA} $1 = \ind{W_1} + \ind{W_1^c}$
\item \label{item:mult} $\ind{W_1 \cap W_2} = \ind{W_1} \ind{W_2}$ 
\item \label{item:or} $\ind{W_1 \cup W_2} = \ind{W_1} + \ind{W_2} - \ind{W_1 \cap W_2}$ 
\item \label{item:redun} If $W_1$ implies $W_2$, then $\ind{W_1}\ind{W_2} = \ind{W_1}$, and we say that $\ind{W_2}$ is redundant. 
\item \label{item:AnotB} If $W_1$ implies $W_2^c$, then $\ind{W_1}\ind{W_2} = 0$. 
\item \label{item:eval} If $W_1$ is a logical statement such that $W_1$ implies $x = x_0$, then for any function $f$ with $x_0$ in the domain of $f$, we have $\ind{W_1}f(x) = \ind{W_1}f(x_0)$. 
\end{enumerate}

\end{lemma}

\begin{proof}
We prove part \eqref{item:AnotA} by cases. If $W_1$ is true, then $W_1^c$ is false and we have $\ind{W_1} + \ind{W_1^c} = 1 + 0 = 1$. The case where $W_1$ is false follows similarly.

The proofs of \eqref{item:mult} through \eqref{item:AnotB} can be proven similarly. To prove \eqref{item:eval} in a similar fashion, we first set $g(x) = \ind{W_1}f(x) - \ind{W_1}f(x_0) = \ind{W_1}\left(f(x) - f(x_0)\right)$. One can then use cases to show $g(x) = 0$ for all $x$, which returns the desired result.
\end{proof}

\section{Free Resolutions for Products of Certain Ideals}

We next consider a construction that takes the minimal resolutions of $R/\mI$ and $R/\mJ$ over $R$ and outputs the minimal resolution of $R/ \mI \mJ$. We then specialize the construction in a corollary to reproduce the minimal, cellular resolution constructed in \cite{MR2262383} for fiber products of the form $k[\vv{x}] \times_k k[\vv{y}] \cong k[\vv{x}, \vv{y}]/\langle \vv{x} \vv{y} \rangle$.


\begin{construct}\label{cons}
Let $\mX$ and $\mY$ be complexes of free $R$-modules. The {\bf star product} of $\mX$ and $\mY$ over $R$, denoted $\amput{\mY}$, is the chain complex given by
\begin{align*}
\left(\amput{\mY} \right)_n = \begin{cases} \left(\mX_{\geq 1} \otimes_R \mY_{\geq 1}\right)_{n+1} & n\geq 1 \\ \mX_0 \otimes_R \mY_0 & n=0 \\ 0 & n < 0 \end{cases} & & \hbox{ and } & & \partial_n^{\mX \mprod \mY} = \begin{cases} \partial_{n+1}^{\mX_{\geq 1} \otimes \mY_{\geq 1}} & n \geq 2 \\ \partial_1^{\mX}\otimes \partial_1^{\mY} & n = 1 \\ 0 & n \leq 0 \end{cases}. 
\end{align*}
\end{construct}

In summary, $\amput{\mY}$ is obtained by truncating $\mX$ and $\mY$, tensoring the truncations, then shifting and augmenting the tensor product. In particular, it is straightforward to show that $\amput{\mY}$ is a bounded below complex of free $R$-modules. We denote simple tensors of positive degree in $\amput{\mY}$ as $\alpha * \beta$ where $\alpha \in \mX$ and $\beta \in \mY$.

\begin{example}\label{ex:star}
Set $\mX = K^{k[\vv{x},\vv{y}]}(\vv{x})$ be the Koszul complex for $\vv{x} = x_1, \ldots, x_m$. Similarly, set $\mY = K^{k[\vv{x},\vv{y}]}(\vv{y})$ be the Koszul complex for $\vv{y} = y_1, \ldots, y_n$. Then, for $R = k[\vv{x},\vv{y}]$, we have $\left(\amput{\mY} \right)_0 \cong k[\vv{x},\vv{y}]$ and, for $\ell \geq 1$, \[\left(\amput{\mY} \right)_{\ell} = \bigoplus_{t = 1}^{\ell} k[\vv{x}, \vv{y}]^{\binom{m}{t}}\otimes_R k[\vv{x},\vv{y}]^{\binom{n}{\ell + 1 - t}} \cong k[\vv{x}, \vv{y}]^{\binom{m + n}{\ell + 1} - \binom{m}{\ell + 1} - \binom{n}{\ell + 1}}.\]

In particular, if we set $R= k[\vv{x},\vv{y}]$ and consider $m = 2 = n$, then $\amput{\mY}$ has the following form.
\[
\xymatrix{
0 \to R \ar[r]^-{\left(\begin{smallmatrix} - x_2 \\ x_1 \\ -y_2 \\ y_1 \end{smallmatrix}\right)} & R^4 \ar[rrrr]^-{\left(\begin{smallmatrix} y_2 & 0 & -x_2 & 0 \\ -y_1 & 0 & 0 & -x_2 \\ 0 & y_2 & x_1 & 0 \\ 0 & -y_1 & 0 & x_1 \end{smallmatrix}\right)} &&&& R^4 \ar[rrrr]^-{\left(\begin{smallmatrix} x_1y_1 & x_1y_2 & x_2y_1 & x_2y_2 \end{smallmatrix}\right)} &&&& R \to 0
}
\]
\end{example}



The following lemma allows to express the differential $\displaystyle \partial^{\amput{\mY}}$ in terms of $\partial^{\mX}$, $\partial^{\mY}$, and indicator functions. In particular, it removes the need to call on the truncated complexes and will simplify later proofs.

\begin{lemma}\label{lem:mdif}
For $n\geq 1$, the differential $\partial^{\mX \mprod \mY}$ acts on $\left(\amput{\mY} \right)_n$ by 
\begin{align*}
\Mdif{ a \mprod b} =& \ind{|a| > 1} \Xdif{a} \mprod b + \ind{|b| > 1}(-1)^{|a|} a \mprod \Ydif{b} + \ind{|a|,|b| = 1} \Xdif{a} \mprod \Ydif{b}.
\end{align*}

\end{lemma}

\begin{proof}
Since $n$ is a positive integer, we have $\ind{n \neq 1} = \ind{|a| + |b| - 1 \geq 2}$. Moreover, we have $n=1$ if and only if $|a| = 1 = |b|$, thus $\ind{n = 1} = \ind{|a|,|b| = 1}$. Thus, parts \eqref{item:AnotA} and \eqref{item:eval} of Lemma \ref{lem:ind} allow us to conclude \[\partial_n^{\mX \mprod \mY}\hspace{-2pt}\left( a \mprod b\right) = \ind{|a| + |b| \geq 3}\hspace{-3pt}\left(\partial_{|a|}^{\mX_{\geq 1}}\hspace{-2pt}(a) \mprod b  + (-1)^{|a|} a \mprod \partial_{|b|}^{\mY_{\geq 1}}\hspace{-2pt}(b) \right) + \ind{|a|,|b| = 1} \Xdif{a} \mprod \Ydif{b}.\]

Lemma \ref{lem:ind}\eqref{item:AnotA}, \eqref{item:AnotB}, and \eqref{item:eval} implies that $\partial_{|a|}^{\mX_{\geq 1}}(a) \mprod b = \ind{|a| > 1}\Xdif{a} \mprod b$ and $a \mprod \partial_{|b|}^{\mY_{\geq 1}}(b) = \ind{|b| > 1} a \mprod \Ydif{b}$. Using Lemma \ref{lem:ind}\eqref{item:redun}, we note that $\ind{|a| + |b| \geq 3}$ is redundant in the presence of both $\ind{|a| > 1}$ and $\ind{|b| > 1}$ and can then be dropped to produce the desired formula.
\end{proof}

Revisiting Example \ref{ex:star}, one can check that the complex given in the case of $m,n = 2$ is in fact a free resolution of $k[\vv{x},\vv{y}]/\langle \vv{x}\vv{y}\rangle$ over $k[\vv{x},\vv{y}]$. This matches with the construction by Visscher in \cite{MR2262383}, which is a particular case in the following theorem. It should also be noted that the following theorem was proved independently and simultaneously by VandeBogert in \cite{vandebogert2020vanishing}.

\begin{theorem}\label{thm:resn}
Let $\mX$ and $\mY$ be free resolutions of $R / \mI$ and $R / \mJ$ over $R$, respectively. If $\mI,\mJ \subseteq R$ are $\operatorname{Tor}$-independent ideals, then the star product $\amput{\mY}$ is a free resolution of $R / \mI\mJ $ over $R$. Moreover, if $\mX$ and $\mY$ are minimal, then $\amput{\mY}$ is also minimal.
\end{theorem}

\begin{proof}
To verify the first conclusion, we need only show that the construction is acyclic and resolves the desired ring. For exactness in degrees $n \geq 2$, we observe that
\begin{align*}
\Hm{n}{\amput{\mY}} &= \Hm{n+1}{\mX_{\geq 1} \otimes_R \mY_{\geq 1}} \\
 &= \Tor{n-1}{\mI}{\mJ} \\
 &\cong \Tor{n+1}{\frac{R}{\mI}}{\frac{R}{\mJ}} \\
 &= 0.
\end{align*}

To treat the case of $n = 1$, we consider the commutative diagram
\[
\xymatrix{
\mX_1 \otimes_R \mY_1 \ar[rr]^{\partial_1^{\amput{\mY}}} \ar[rd]^{\partial_1^{\mX} \otimes \partial_1^{\mY}} && \mI\mJ \\
 & \IM \partial_1^{\mX} \otimes_R \IM \partial_1^{\mY} = \mI \otimes_R \mJ \ar[ru]_-{\mu} &
}\]
and observe that $\ker \left( \partial_1^{\mX} \otimes \partial_1^{\mY} \right) \subseteq \ker \partial_1^{\amput{\mY}}$ with equality if the map $\mu$ sending $\alpha \otimes \beta \in \mI \otimes_R \mJ$ to $\alpha\beta \in \mI \mJ$ is injective. To show that $\mu$ is injective, consider the long exact sequence in Tor obtained by applying $\mI \otimes_R -$ to the short exact sequence \[0 \to \mJ \to R \to R/\mJ \to 0.\] This yields the exact sequence \[ 0 \to \underbrace{\Tor{1}{\mI}{\frac{R}{\mJ}}}_{\cong \Tor{2}{\frac{R}{\mI}}{\frac{R}{\mJ}}=0} \to \mI \otimes_R \mJ \to \underbrace{\mI \otimes_R R}_{\cong \mI} \to \underbrace{\mI \otimes_R R/\mJ}_{\cong \mI/\mI\mJ} \to 0\] from which we deduce the exact sequence \[\xymatrix{0 \ar[r] & \mI \otimes_R \mJ \ar[r]^-{\mu} & \mI\mJ \ar[r] & 0.}\] Thus, $\mu$ is an isomorphism so $\ker \left( \partial_1^{\mX} \otimes \partial_1^{\mY} \right) = \ker \partial_1^{\amput{\mY}}$.

We note that $\partial_1^{\mX}: \mX_1 \to \IM \partial_1^{\mX}$ and $\partial_1^{\mY}: \mY_1 \to \IM \partial_1^{\mY}$ are both surjections. It follows that \[
\xymatrix{
\mX_1 \otimes_R \mY_1 \ar[rr]^-{\partial_1^{\mX} \otimes \partial_1^{\mY}} & & \IM \partial_1^{\mX} \otimes_R \IM \partial_1^{\mY} \ar[r] & 0
} \] is exact. Furthermore, \cite[p. 267]{MR1011461} shows us that kernel of this surjection is given by 
\begin{align*}
\ker \partial_1^{\mX \mprod \mY} &= \ker \left( \partial_1^{\mX} \otimes \partial_1^{\mY} \right)\\
 & = \ker \partial_1^{\mX} \otimes_R \mY_1 + \mX_1 \otimes_R \ker \partial_1^{\mY} \\
 & = \IM \partial_2^{\mX} \otimes_R \mY_1 + \mX_1 \otimes_R \IM \partial_2^{\mY} \\
 & = \IM \partial_2^{\mX \mprod \mY}.
\end{align*}
Moreover, we note \[H_0\left(\amput{\mY}\right) = \frac{\mX_0 \otimes_R \mY_0}{\IM\left(\partial_1^{\mX} \otimes \partial_1^{\mY}\right)}  \cong \frac{R \otimes_R R}{\mI \otimes_R \mJ} \cong \frac{R}{ \mI \mJ }. \]
Lastly, it is straightforward to check minimality.
\end{proof}

As we discuss in the introduction, setting $\mI = \langle \vv{x} \rangle$ and $\mJ = \langle \vv{y} \rangle$ in the case of $R = k[\vv{x},\vv{y}]$ in Theorem \ref{thm:resn} recovers the main result of Visscher in \cite{MR2262383}. Next, we document consequences for Betti numbers.

\begin{cor} \label{cor:bettiprod}
We have the following formulas for the Betti numbers and graded Betti Numbers of $R/ \mI \mJ$ for $\ell > 0$.
\begin{align*}
 \beta_{\ell}^R(R/\mI \mJ ) &=  \sum_{i=1}^{\ell}\beta_i^R(R/\mI) \beta_{\ell + 1 - i}^R(R/\mJ) \\
 \beta_{\ell,k}^R(R/ \mI \mJ)  &= \sum_{i = 1}^{\ell}\sum_{j = 0}^k \beta_{i,j}^R(R/\mI)\beta_{\ell + 1 - i, k - j}^R(R/\mJ)
\end{align*}
\end{cor}





\begin{proof}
Use the fact that $\amput{\mY}$ is minimal and obtained by truncating, tensoring, shifting, and augmenting minimal resolutions to get the desired formulas. \end{proof}

We conclude this section by interpreting Corollary \ref{cor:bettiprod} in terms of Poincar\'e series.

\begin{cor}\label{cor:pon}
We have the following relationship of Poincar\'{e} series: \[ P_{R/\mI \mJ}^{R}(t) = 1 + \frac{1}{t}(P_{R/\mI}^{R}(t) - 1)(P_{R/\mJ}^{R}(t) - 1).\]
\end{cor}






\section{Extending the Star Product}

In this section, we use $\amput{\mY}$ to build a resolution of $R/\langle \mI', \mI\mJ, \mJ'\rangle$ over $R$ where $\mI' \subseteq \mI$ and $\mJ' \subseteq \mJ$. Taking $R = k[[\vv{x},\vv{y}]]$, we will show how this specializes to fiber products of the form $k[[\vv{x}]]/I' \times_k k[[\vv{y}]]/J'$ and note the same can be done in the polynomial case.


\begin{notation}\label{sec4notation}
Throughout this section, we fix the following notation.
	\begin{enumerate}
		\item Let $\mX$ resolve $R/\mI$ over $R$.
		\item Let $\mY$ resolve $R/\mJ$ over $R$.
		\item Let $\cS$ resolve $R/\mI'$ over $R$.
		\item Let $\cT$ resolve $R/ \mJ'$ over $R$.
		\item \label{item:TorInd} The pairs $\{\mI,\mJ\}$, $\{\mI', \mJ\}$, and $\{\mI, \mJ'\}$ are $\operatorname{Tor}$-independent ideals in $R$.
	\end{enumerate}
\end{notation}

The next example shows that some ideals from Example \ref{ex:fiber} are Tor-independent for use in our proof of Theorem \ref{thm:main}.

\begin{example}\label{Thm:TorInd}
Suppose that we either have one of the following cases.
\begin{enumerate}
	\item $A = k[\vv{x}]$ and $B = k[\vv{y}]$ with $R = k[\vv{x},\vv{y}]$.
	\item $A = k[[\vv{x}]]$ and $B = k[[\vv{y}]]$ with $R = k[[\vv{x},\vv{y}]]$.
\end{enumerate}
If $M$ is an $A$-module and $N$ is a $B$-module, both finitely generated, then a standard prime filtration argument using induction on dimension shows that $M \otimes_A R$ and $N \otimes_B R$ are Tor-independent $R$-modules.
\end{example}

The Tor-independence in Notation \ref{sec4notation}\eqref{item:TorInd} is needed to ensure the vanishing of homology in Construction \ref{cons} and in the following lemma.

\begin{lemma}\label{lem:ideal}
Under the assumptions in Notation \ref{sec4notation}, we have $\Sigma^{-1}(\cS_{\geq 1} \otimes_R \mY)$ resolves $\mI' \cdot R/ \mI\mJ $ over $R$ while $\Sigma^{-1}(\mX \otimes_R \cT_{\geq 1})$ resolves $\mJ' \cdot R/ \mI\mJ$.
\end{lemma}

\begin{proof}
We prove the case of $\Sigma^{-1}(\cS_{\geq 1} \otimes_R \mY)$. Since $\mI'$ and $\mJ$ are $\operatorname{Tor}$-independent, we have $\mI'\cap \mJ = \mI' \mJ$. Combining this with the fact that $\mI' \subseteq \mI$, we have \[\mI'\mJ \subseteq \mI' \cap \mI\mJ \subseteq \mI' \cap \mJ = \mI' \mJ\] yielding $\mI'\mJ = \mI' \cap \mI\mJ$. Thus, we have the following chain of isomorphisms: \[\mI' \cdot \frac{R}{\mI \mJ} = \frac{\mI' + \mI \mJ}{\mI \mJ} \cong \frac{\mI'}{\mI' \cap \mI \mJ} = \frac{\mI'}{\mI' \mJ} \cong \mI' \otimes_R \frac{R}{\mJ}.\]

It is straightforward to check that $\Hm{0}{\Sigma^{-1}\left(\cS_{\geq 1} \otimes_R \mY\right)} = \mI' \otimes_R R/\mJ$. For $i \geq 1$, we use the $\operatorname{Tor}$-independence of $\mI'$ and $\mJ$ to get \[\Hm{i}{\Sigma^{-1}\left(\cS_{\geq 1} \otimes_R \mY\right)} = \Tor{i}{\mI'}{\frac{R}{\mJ}} \cong \Tor{i+1}{\frac{R}{\mI'}}{\frac{R}{\mJ}} = 0.\] Thus, $\Sigma^{-1}(\cS_{\geq 1} \otimes_R \mY)$ is the desired resolution. \end{proof}

\begin{note}\label{note:lift}
Since $\mI' \subseteq \mI$, we have a natural surjection $R/\mI' \to R/\mI$ which lifts to a chain map $\phi: \cS \to \mX$. We also lift the surjection $R/\mJ' \to R/\mJ$ to a chain map $\psi: \cT \to \mY$. In the next construction, we lift $\phi$ and $\psi$ to chain maps $\Phi:~\Sigma^{-1}\left(\cS_{\geq 1} \otimes_R \mY\right) \to \amput{\mY}$ and $\Psi: \Sigma^{-1}\left(\mX \otimes_R \cT_{\geq 1}\right) \to \amput{\mY}$.\end{note}

\begin{construct}\label{con:lift}
For $\alpha \otimes \beta \in \Sigma^{-1}\left(\cS_{\geq 1} \otimes_R \mY\right)$, we define
\begin{align*}
\Phi\left(\alpha \otimes \beta\right) &:= \begin{cases} (-1)^{|\alpha| + |\beta|} \phi\left(\alpha\right) \mprod \beta & |\beta| > 0 \\ \partial_1^{\cS}\left(\alpha\right) \mprod \beta & |\beta| = 0, \ |\alpha| = 1 \\ 0 & |\beta| = 0, \ |\alpha| > 1 \end{cases} \\
 &= \ind{|\beta| > 0}(-1)^{|\alpha| + |\beta|} \phi\left(\alpha\right) \mprod \beta + \ind{|\beta| = 0} \ind{|\alpha| = 1} \partial_1^{\cS}\left( \alpha \right) \mprod \beta.
\end{align*}
In the case that $\alpha \otimes \beta \in \Sigma^{-1}\left(\mX \otimes_R \cT_{\geq 1}\right)$, we define
\begin{align*}
\Psi(\alpha \otimes \beta) &:= \begin{cases} (-1)^{|\alpha| + |\beta| - 1} \alpha \mprod \psi(\beta) & |\alpha| > 0 \\ \alpha \mprod \partial_1^{\cT}(\beta) & |\alpha| = 0, |\beta| = 1 \\ 0 & |\alpha| = 0, |\beta| > 1 \end{cases} \\
 &= -\ind{|\alpha| > 0}(-1)^{|\alpha| + |\beta|}\alpha \mprod \psi(\beta) + \ind{|\alpha| = 0}\ind{|\beta| = 1} \alpha \mprod \partial_1^{\cT}(\beta).
\end{align*}
The second equality in both definitions follow as in the proof of Lemma \ref{lem:mdif}.
\end{construct}

\begin{lemma}\label{lem:chmap}
The maps $\Phi$ and $\Psi$, as defined in Construction \ref{con:lift}, are chain maps.
\end{lemma}

\begin{proof}
For the $\Phi$ computations, we use the following formula
\begin{align}\label{align:chain1}
\partial_i^{\Sigma^{-1}\left(\cS_{\geq 1} \otimes_R \mY\right)}\left(\alpha \otimes \beta\right) =& -\partial_{i+1}^{\cS_{\geq 1} \otimes_R \mY}\left(\alpha \otimes \beta\right) \notag \\
 =& -\left(\partial^{\cS_{\geq 1}}(\alpha) \otimes \beta + (-1)^{|\alpha|}\alpha \otimes \Ydif{\beta}\right) \\
 \stackrel{\ref{lem:ind}\eqref{item:eval}}{=}& -\ind{|\alpha| > 1}\partial^{\cS}(\alpha) \otimes \beta - (-1)^{|\alpha|} \alpha \otimes \Ydif{\beta}. \notag
\end{align}
Note that the first two equalities follow by definition.

To see that $\Phi$ is a chain map, we use Lemma \ref{lem:ind}\eqref{item:redun} to get $\ind{|\beta| = 1} = \ind{|\beta| = 1} \ind{ |\beta| > 0}$ and $\ind{|\beta| > 1} = \ind{|\beta| > 1} \ind{ |\beta| > 0}$. Since $\phi$ is a chain map with $\phi_0 = 1_{k[\vv{x}]}$, we have $\partial_1^{\cS}(\alpha) = \Xdif{\phi(\alpha)}$. Lastly, we note that Lemma \ref{lem:ind}\eqref{item:eval} yields
\begin{align}\label{align:van}
\ind{|\beta| = 0}\ind{|\alpha| = 1} \Mdif{ \partial_1^{\cS}(\alpha) \mprod \beta} &=  0.
\end{align}
 Thus, for all $i \in \ZZ$ we have
\begin{align*}
\left(\Phi_{i-1} \circ \partial_i^{\Sigma^{-1}\left(\cS_{\geq 1} \otimes_R \mY\right)}\right)\left(\alpha \otimes \beta \right) \hspace{-3cm} \\
 \stackrel{\ref{align:chain1}}{=}& -\ind{|\alpha| > 1}\Phi_{i-1}\left(\partial^{\cS}(\alpha) \otimes \beta\right) - (-1)^{|\alpha|} \Phi_{i-1}\left(\alpha \otimes \Ydif{\beta}\right) \\
 \stackrel{\ref{con:lift}}{=}& -\ind{|\alpha| > 1} \ind{|\beta| > 0} (-1)^{|\alpha| - 1 + |\beta|} \left( \phi \circ \partial^{\cS} \right)(\alpha) \mprod \beta \\
 & \qquad - \ind{\left| \Ydif{\beta} \right| > 0} (-1)^{2|\alpha| + |\beta| - 1} \phi(\alpha) \mprod \Ydif{\beta}   \\
 & \qquad - \ind{\left|\Ydif{\beta}\right| = 0}\ind{|\alpha| = 1} (-1)^{|\alpha|} \partial^{\cS}(\alpha) \mprod \Ydif{\beta} \\
 \stackrel{*}{=}& \ \ind{|\alpha| > 1} \ind{|\beta| > 0} (-1)^{|\alpha| + |\beta|} \left( \partial^{\mX} \circ \phi \right)(\alpha) \mprod \beta \\
 & \qquad  + \ind{|\beta| > 1} (-1)^{2|\alpha| + |\beta| } \phi(\alpha) \mprod \Ydif{\beta}   \\
 & \qquad + \ind{|\beta| = 1}\ind{|\alpha| = 1} (-1)^{|\alpha| + |\beta|} \left(\partial^{\mX} \circ \phi\right) (\alpha) \mprod \Ydif{\beta} \\
 \stackrel{\ref{lem:ind}\eqref{item:redun}}{=} & \ind{|\alpha| > 1} \ind{|\beta| > 0} (-1)^{|\alpha| + |\beta|} \left( \partial^{\mX} \circ \phi \right)(\alpha) \mprod \beta \\
 & \qquad  + \ind{|\beta| > 1}\ind{|\beta| > 0} (-1)^{2|\alpha| + |\beta| } \phi(\alpha) \mprod \Ydif{\beta}   \\
 & \qquad + \ind{|\beta| = 1}\ind{|\beta| > 0}\ind{|\alpha| = 1} (-1)^{|\alpha| + |\beta|} \left(\partial^{\mX} \circ \phi\right) (\alpha) \mprod \Ydif{\beta} \\
 \stackrel{\ref{lem:mdif}}{=}& \ind{|\beta| > 0} (-1)^{|\alpha| + |\beta|} \Mdif{\phi(\alpha) \mprod \beta} \\
 \stackrel{\ref{align:van}}{=}& \ind{|\beta| > 0} (-1)^{|\alpha| + |\beta|} \Mdif{\phi(\alpha) \mprod \beta} + \ind{|\beta| = 0}\ind{|\alpha| = 1} \Mdif{\partial_1^{\cS}(\alpha) \mprod \beta} \\
 =& \ \Mdif{\ind{|\beta| > 0} (-1)^{|\alpha| + |\beta|} \phi(\alpha) \mprod \beta + \ind{|\beta| = 0}\ind{|\alpha| = 1} \partial_1^{\cS}(\alpha) \mprod \beta} \\
 \stackrel{\ref{con:lift}}{=}& \left(\partial_i^{\amput{\mY}} \circ \Phi_i \right) \left(\alpha \otimes \beta\right).
\end{align*}
The unmarked equality follows from the linearity of $\partial^{\amput{\mY}}$ while starred equality follows from the fact that $\phi$ is a chain map and that $|\Ydif{b}| = |b| - 1$. The proof that $\Psi$ is a chain map is similar.
\end{proof}

The following result is a special case of Theorem \ref{thm:main}.
\begin{theorem}\label{thm:cones}
The complex $\Cone{\Phi}$ resolves $R/ \langle \mI', \mI\mJ \rangle$ over $R$, and $\Cone{\Psi}$ resolves $R / \langle \mI\mJ, \mJ'\rangle$ over $R$. 
\end{theorem}

\begin{proof}
Since $\Sigma\left(\Sigma^{-1} \left(\cS_{\geq 1} \otimes_R \mY\right)\right) = \cS_{\geq 1}\otimes_R \mY$, the short exact sequence of complexes associated with the mapping cone is given by
\[
\xymatrix{
0 \ar[r] & \amput{\mY} \ar[r] & \Cone{\Phi} \ar[r] & \cS_{\geq 1}  \otimes_R \mY \ar[r] & 0.
}
\]
Since $\amput{\mY}$ only has homology in degree 0 and $\cS_{\geq 1} \otimes_R \mY$ only has homology in degree 1, the corresponding long exact sequence in homology reduces to the top row in the following commutative diagram. Moreover, the Short-Five Lemma tells us that the diagram is an isomorphism of short exact sequences.
\[
\xymatrix{
 0 \ar[r] & \Hm{1}{\cS_{\geq 1} \otimes_R \mY} \ar[r] \ar[d]^{\cong} & \Hm{0}{\amput{\mY}} \ar[r] \ar[d]^{\cong} & \Hm{0}{\Cone{\Phi}} \ar[r] \ar@{-->}[d]^{\cong} & 0 \\
 0 \ar[r] & \mI' \cdot \frac{R}{\mI\mJ} \ar[r] & \frac{R}{\mI\mJ} \ar[r] & \frac{R}{\langle \mI', \mI\mJ \rangle} \ar[r] & 0
}
\]

The argument for $\Cone{\Psi}$ follows nearly identical arguments. \end{proof}

In the next result we use a short exact sequence of complexes to intertwine the mapping cones from Theorem \ref{thm:cones}.

\begin{theorem}\label{thm:genres}
Let $\Phi$ and $\Psi$ be as in construction \ref{con:lift}, set \[\Omega = \begin{pmatrix} \Phi & \Psi \end{pmatrix} : \Sigma^{-1}\left(\cS_{\geq 1} \otimes_R \mY\right) \oplus \Sigma^{-1}\left(\mX \otimes_R \cT_{\geq 1}\right) \longrightarrow \amput{\mY}.\] Then $\Cone{\Omega}$ is a free resolution of $R /\langle \mI', \mI\mJ, \mJ' \rangle$ over $R$.
\end{theorem}

\begin{proof}
By definition of $\Omega$, it is straightforward to show that the following sequence is exact:
\[
\xymatrix{
0 \ar[r] & \amput{\mY} \ar[r]^-{\left(\begin{smallmatrix} 1 \\ 0 \\ -1 \\ 0 \end{smallmatrix}\right)} & \underset{\Cone{\Psi}}{\overset{\Cone{\Phi}}{\bigoplus}} \ar[rr]^-{\left(\begin{smallmatrix} 1 & 0 & 1 & 0 \\ 0 & 1 & 0 & 0 \\ 0 & 0 & 0 & 1 \end{smallmatrix}\right)} && \Cone{\Omega} \ar[r] & 0.
}
\]
We use the associated long exact sequence in homology to see that $\Hm{i}{\Cone{\Omega}} = 0$ for $i > 1$, leaving the case $i = 1$.
\begin{align*}
\Hm{1}{\Cone{\Omega}} \hspace{-1cm} \\
 &\cong \ker\left[\Hm{0}{\amput{\mY}} \to \Hm{0}{\Cone{\Phi}}\right] \cap \ker\left[\Hm{0}{\amput{\mY}} \to \Hm{0}{\Cone{\Psi}}\right] \\
 &\cong \ker\left[\frac{R}{\mI \mJ} \to \frac{R}{\langle \mI', \mI\mJ \rangle}\right] \cap \ker\left[\frac{R}{\mI \mJ} \to \frac{R}{\langle \mI \mJ, \mJ' \rangle}\right] \\
 &= \left(\mI' \cdot \frac{R}{\mI \mJ} \right) \cap \left(\mJ' \cdot \frac{R}{\mI \mJ}\right)
\end{align*} This intersection is 0 because the assumptions $\mI' \subseteq \mI$ and $\mJ' \subseteq \mJ$, along with Tor-independence, imply that \[(\mI' + \mI\mJ) \cap (\mI\mJ + \mJ') \subseteq \mI \cap \mJ = \mI\mJ\] so we get $\Hm{1}{\Cone{\Omega}} = 0$.

The long exact sequence in homology now reduces to the top row of the following commutative diagram.
\[
\xymatrix{
0 \ar[r] & \Hm{0}{\amput{\mY}} \ar[r]^{\overline{\left(\begin{smallmatrix} 1 \\ 0 \\ -1 \\ 0 \end{smallmatrix}\right)}} \ar[d]^{\cong} & \underset{\Hm{0}{\Cone{\Psi}}}{\overset{\Hm{0}{\Cone{\Phi}}}{\bigoplus}} \ar[rr]^-{\overline{\left(\begin{smallmatrix} 1 & 0 & 1 & 0 \\ 0 & 1 & 0 & 0 \\ 0 & 0 & 0 & 1 \end{smallmatrix}\right)}} \ar[d]^{\cong} && \Hm{0}{\Cone{\Omega}} \ar[r] \ar@{-->}[d]^{\cong} & 0 \\
0 \ar[r] & \frac{R}{\mI \mJ} \ar[r]_-{\left(\begin{smallmatrix} 1 \\ -1 \end{smallmatrix}\right)} & \underset{R/\langle \mI \mJ, \mJ' \rangle }{\overset{R /\langle \mI', \mI \mJ \rangle}{\bigoplus}} \ar[rr]_-{\left(\begin{smallmatrix} 1 & 1 \end{smallmatrix}\right)} && \frac{R}{\langle \mI', \mI \mJ,\mJ' \rangle} \ar[r] & 0
}
\] Thus, $\Cone{\Omega}$ resolves the desired quotient. \end{proof}




\section{Minimality Conditions}

To make full use of the mapping cone in Theorem \ref{thm:genres} we want to ensure it is minimal. We saw in Theorem \ref{thm:resn} that if the resolutions $\mX$ and $\mY$ are minimal, then $\amput{\mY}$ is minimal. In this section, we give conditions guaranteeing $\Cone{\Omega}$ from Theorem \ref{thm:genres} is minimal.

\begin{lemma}\label{lem:kosim}
Using the notation of \ref{sec4notation}, let $\mI$ be generated by a regular sequence with minimal resolution $\mX$ with underlying graded module $\mX^{\natural}$. If $M \subseteq \mX^{\natural}$ is a submodule such that $\Xdif{M} \subseteq \left(\mI\right)^2 \mX$, then $M \subseteq \mI \mX$.
\end{lemma}

\begin{proof}
Since $\mI$ is generated by a regular sequence, we know that $\mX$ is a Koszul complex. For any $m \in M$, we have $\Xdif{m} \in \left(\mI\right)^2 \mX$. By \cite[Lemma 2.8(ii)]{MR3754407}, we have $m \in \mI \mX$ and thus $M \subseteq \mI \mX$.\end{proof}

\begin{prop}\label{prop:regseq}
With the notation of \ref{sec4notation}, suppose the ideal $\mI \subset R$ is generated by a regular sequence and that $\mI' \subseteq \mI^2$. If $\mX$ and $\cS$ are minimal, then $\phi: \cS \to \mX$ from Note \ref{note:lift} can be chosen such that $\phi_0 = 1$ and $\left(\IM \phi\right)_{\geq 1} \subseteq \mI \mX \subseteq \maxM_R \mX$.
\end{prop}

\begin{proof}

Since $\mI' \subseteq \mI^2$, the natural surjection $R/\mI' \to R/\mI$ can be factored through $R/ \mI^2$, so it suffices to check that $R/ \mI^2 \to R/\mI$ lifts to a chain map $\phi$ with the claimed property. Thus, we assume without loss of generality that $\mI' = \mI^2$. Since $\cS$ and $\mX$ resolve powers of $\mI$, \cite[Corollary~A2.13, Exercise A2.17(d)]{MR1322960} tells us that $\IM \partial_i^{\cS} \subseteq \mI \cS$ and $\IM \partial_i^{\mX} \subseteq \mI \mX$ for all $i$.

Consider $\IM \phi_1 \subseteq \mX$ and note that $\partial_1^{\cS}\left(\cS_1\right) = \mI^2$. It follows that \[\Xdif{\IM \phi_1} = \Xdif{\phi_1\left(\cS_1\right)} = \phi_0\left(\partial^{\cS}\left(\cS_1\right)\right) = \phi_0\left(\mI^2\right) = \mI^2. \] It follows from Lemma \ref{lem:kosim} that $\IM \phi_1 \subseteq \mI \mX_1$. Proceeding with induction, we observe that if $\IM \phi_{j-1} \subseteq \mI \mX_{j-1}$, then \[\Xdif{\IM \phi_j} = \Xdif{\phi_j\left(\cS_j\right)} = \phi_{j-1}\left(\partial^{\cS}\left(\cS_j\right)\right) \subseteq \mI \phi_{j-1}\left(\cS_{j-1}\right) \subseteq \left(\mI\right)^2 \mX_{j-1}.\] Applying Lemma \ref{lem:kosim} again, we obtain $\IM \phi_j \subseteq \mI \mX_j$. Thus by induction, we observe that $\left(\IM \phi\right)_{\geq 1} \subseteq \mI \mX$. \end{proof}

We now use this proposition to obtain the following theorem.

\begin{theorem}\label{thm:minimal}
Suppose $\mI' \subseteq \mI^2$ and $\mJ' \subseteq \mJ^2$ with both $\mI$ and $\mJ$ generated by regular sequences. If the resolutions $\mX$, $\mY$, $\cS$, and $\cT$ are all minimal and we choose $\phi$ and $\psi$ as in Proposition \ref{prop:regseq}, then $\Cone{\Omega}$ from Theorem \ref{thm:genres} is minimal.
\end{theorem}


\begin{proof}
To prove minimality we just need to show that $\IM \partial^{\Cone{\Omega}} \subseteq \maxM_R \Cone{\Omega}$. To see this, we first note that \[\partial_i^{\Cone{\Omega}} = \begin{pmatrix} \partial_i^{\amput{\mY}} & \Phi_{i-1} & \Psi_{i-1} \\ 0 & \partial_i^{\cS_{\geq 1} \otimes_R \mY} & 0 \\ 0 & 0 & \partial_i^{\mX \otimes_R \cT_{\geq 1}} \end{pmatrix}.\] By the minimality of our four resolutions, we know the entries on the main diagonal all lie in $\maxM_R$. Thus, we need only address $\Phi_{i-1}$ and $\Psi_{i-1}$. We are argue the case of $\Phi_{i-1}$.

For any $\alpha \otimes \beta$ in the domain of $\Phi$, we have $\alpha \in \cS_{\geq 1}$. It follows that $\phi(\alpha) \in \mI \mX$ and thus $\phi(\alpha) \mprod \beta \in \mI \left(\amput{\mY}\right) \subseteq \maxM_R \left(\amput{\mY}\right)$. Moreover, we know $\partial^{\cS}(\alpha) \in \mI \cS$ and thus $\partial^{\cS}(\alpha) \mprod \beta \in \mI \left(\amput{\mY}\right) \subseteq \maxM_R \left(\amput{\mY}\right)$. Combining these observation with the definition of $\Phi$ in Construction \ref{con:lift} we have $\Phi(\alpha\otimes \beta) \in \maxM_R \left(\amput{\mY}\right)$ for all $\alpha \otimes \beta \in \Sigma^{-1}\left(\cS_{\geq 1} \otimes_R \mY\right)$ as desired.
\end{proof}





\begin{cor}\label{cor:betti}
If $\mI$ is generated by $m$ elements and $\mJ$ by $n$ elements, then for $F := R/\langle \mI', \mI\mJ, \mJ'\rangle$ and $\ell \geq 1$, we have the following formula for Betti numbers.
\begin{align}\label{eqn:desbetti}
\beta_{\ell}^{R}\left( F \right) =& \sum_{t = 1}^{\ell}\left(\beta_t^{R}\left( R/\mI' \right)\binom{n}{\ell - t} + \binom{m}{\ell - t}\beta_t^{R}\left( R/\mJ' \right) \right) \\
 & \qquad + \binom{m + n}{\ell + 1} - \binom{m}{\ell + 1} - \binom{n}{\ell + 1} \notag
\end{align}
In addition, if all the ideals are homogeneous then we get the following formulas for the graded Betti numbers: for $\ell = 0$, we have $\beta_{0,k}^R(F) = \beta_{0,k}^R\left(\frac{R}{\mI\mJ}\right)$, while for $\ell > 0$, we have
\begin{align*}
\beta_{\ell,k}^R(F) =&  \beta_{\ell,k}^R\left(\frac{R}{\mI \mJ}\right) + \sum_{i=1}^{\ell}\sum_{j = 0}^k \beta_{i,j}^R\left(\frac{R}{\mI'}\right)\beta_{\ell - i, k- j}^R\left(\frac{R}{\mJ}\right) \\
 &\qquad +\sum_{i=1}^{\ell}\sum_{j = 0}^k \beta_{i,j}^R\left(\frac{R}{\mI}\right)\beta_{\ell - i, k - j}^R\left(\frac{R}{\mJ'}\right).
\end{align*}
\end{cor}

\begin{proof}
Since $\mI$ and $\mJ$ are generated by regular sequences, we have $\beta_i^R\left(\frac{R}{\mI}\right) = \binom{m}{i}$ and $\beta_i^R\left(\frac{R}{\mJ}\right) = \binom{n}{i}$.



We note that $\Cone{\Omega}^{\natural} \cong \left(\amput{\mY}\right)^{\natural} \oplus \left(\cS_{\geq 1} \otimes_R \mY\right)^{\natural} \oplus \left(\mX \otimes_R \cT_{\geq 1}\right)^{\natural}$ as $R$-modules. Combining this observation with the fact that $\Sigma^{-1}\left(\cS_{\geq 1} \otimes_R \mY\right)$ resolves $\mI' \otimes_R \frac{R}{\mJ}$ and $\Sigma^{-1}\left(\mX \otimes_R \cT_{\geq 1}\right)$ resolves $\frac{R}{\mI} \otimes_R \mJ'$, we have
\begin{align}\label{eqn:betti}
\beta_{\ell}^R(F) &= \beta_{\ell}^R\left(\frac{R}{\mI\mJ}\right) + \beta_{\ell - 1}^R\left(\mI' \otimes_R \frac{R}{\mJ}\right) + \beta_{\ell - 1}^R\left(\frac{R}{\mI} \otimes_R \mJ'\right).
\end{align} For $\ell = 0$, the last two terms vanish and the problem reduces to Corollary \ref{cor:bettiprod}.

For the case of $\ell \geq 1$, we address each of three terms in equation \ref{eqn:betti}. For the first term, we combine Vandermonde's identity \begin{align}\label{eqn:van}
\binom{m + n}{r} = \sum_{t = 0}^r \binom{m}{t}\binom{n}{r-t}.
\end{align} with Corollary \ref{cor:bettiprod} to get
\begin{align*}
\beta_{\ell}^R\left(\frac{R}{\mI\mJ}\right) \stackrel{\ref{cor:bettiprod}}{=}& \sum_{i=1}^{\ell}\beta_i^R \left(\frac{R}{\mI}\right) \beta_{\ell + 1 - i}^R \left(\frac{R}{\mJ}\right) \\
 =& \sum_{i=0}^{\ell + 1} \binom{m}{i} \binom{n}{\ell + 1 - i} - \binom{m}{0}\binom{n}{\ell + 1} - \binom{m}{\ell + 1}\binom{n}{0} \\
 \stackrel{\ref{eqn:van}}{=} & \binom{m + n}{\ell + 1} - \binom{m}{\ell + 1} - \binom{n}{\ell + 1}.
\end{align*}


For the other two terms, we make use of the $\operatorname{Tor}$-independence of our ideals. Moreover, by construction we have
\begin{align*}
\beta_{\ell - 1}^R\left(\mI' \otimes_R \frac{R}{\mJ}\right) = & \sum_{i = 0}^{\ell - 1} \beta_i^R(\mI')\beta_{\ell - 1 - i}^R\left(\frac{R}{\mJ}\right) \\
 = & \sum_{i = 0}^{\ell - 1}\beta_{i + 1}^R\left(\frac{R}{\mI'}\right)\binom{n}{\ell - (i + 1)} \\
 = & \sum_{i = 1}^{\ell}\beta_i^R\left(\frac{R}{\mI'}\right) \binom{n}{\ell - i}.
\end{align*}
The same reasoning is used to obtain the desired formula $\beta_{\ell - 1}^R\left(\frac{R}{\mI} \otimes_R \mJ'\right)$. Substituting these into equation \ref{eqn:betti} yields equation \ref{eqn:desbetti}, which is our desired result.

For the graded Betti numbers, we note that there is an analogous expression of equation \ref{eqn:betti}. For $\ell > 0$, the argument follows the non-graded case using the formula \[\beta_{\ell - 1,k}^R\left(\mI' \otimes_R \frac{R}{\mJ}\right) = \sum_{i = 0}^{\ell - 1} \sum_{j = 0}^k \beta_{i,j}^R(\mI')\beta_{\ell - 1 - i, k - j}^R\left(\frac{R}{\mJ}\right). \] The case of $\ell = 0$ reduces to Corollary \ref{cor:bettiprod}.  \end{proof}


We now obtain the Poincar\'{e} series portion of Theorem \ref{thm:main}. In the case of fiber products, setting $\mI = \langle \vv{x} \rangle$ and $\mJ = \langle \vv{y} \rangle$ with $R = k[\vv{x},\vv{y}]$ (or its completion) recovers \cite[Corollary 5.1.3]{MR2711725}.


\begin{cor}\label{cor:poin}
The Poincar\'{e} series of $F:= R/\langle \mI', \mI \mJ, \mJ' \rangle$ satisfies
\begin{align}\label{eqn:pon1}
\frac{P_{F}^{R}(t) - P_{R/(\mI' + \mJ)}^{R}(t) - P_{R/(\mI + \mJ')}^{R}(t) + P_{R/(\mI + \mJ)}^{R}(t)}{P_{ \mI\mJ }^{R}(t)} = t(1+t)
\end{align}
and
\begin{align}\label{eqn:pon2}
\frac{P_{F}^{R}(t) - (1+t)^n P_{R/\mI'}^{R}(t) - (1+t)^m P_{R/\mJ'}^{R}(t) + (1 + t)^{m + n}}{\left( (1 + t)^m - 1 \right) \left( (1 + t)^n - 1 \right)} = \frac{t + 1}{t}.
\end{align}
\end{cor}

\begin{proof}
The unlabeled equalities in the following display are by definition.
\begin{align*}
P_F^{R}(t) &= \PP_{\amput{\mY}}^{R}(t) + \PP_{\cS_{\geq 1} \otimes_R \mY}^{R}(t) + \PP_{\mX \otimes_R \cT_{\geq 1}}^{R}(t) \\
 &= P_{R/ \mI\mJ }^{R}(t) + \PP_{\cS_{\geq 1}}^{R}(t) \PP_{\mY}^{R}(t) + \PP_{\mX}^{R}(t) \PP_{\cT_{\geq 1}}^{R}(t) \\
 &= 1 + t \cdot P_{ \mI\mJ }^{R}(t) + (P_{R/\mI'}^{R}(t) - 1)P_{R/\mJ}^{R}(t)  + P_{R/\mI}^{R}(t)  (P_{R/\mJ'}^{R}(t) - 1).
\end{align*}

Rearranging and applying Corollary \ref{cor:pon} yields
\begin{align}\label{align:poin}
P_F^{R}(t) - P_{R/\mI'}^{R}(t) P_{R/\mJ}^{R}(t) - P_{R/\mI}^{R}(t) P_{R/\mJ'}^{R}(t) \hspace{-2cm}  \\
 &= 1 + t \cdot P_{ \mI \mJ }^{R}(t)  - P_{R/\mJ}^{R}(t) - P_{R/\mI}^{R}(t). \notag
\end{align}
Our Tor-independence assumptions in Notation \ref{sec4notation} guarantee the following equations
\begin{align}\label{align:denom}
P_{R/\mI'}^{R}(t)P_{R/\mJ}^{R}(t) &= P_{R/(\mI' + \mJ)}^{R}(t) \notag \\
P_{R/\mI}^{R}(t) P_{R/\mJ}^{R}(t) &= P_{R/(\mI + \mJ')}^{R}(t) \notag \\
P_{R/\mI}^R(t) P_{R/\mJ}^{R}(t) &= P_{R/(\mI + \mJ)}^{R}(t) \notag\\
t^2 \cdot P_{ \mI \mJ }^{R}(t) &= \left(P_{R/\mI}^{R}(t) - 1 \right)\left(P_{R/\mJ}^{R}(t) - 1 \right).
\end{align}
Combining these four equations with equation \ref{align:poin}, yields
\begin{align*}
P_F^{R}(t) - P_{R/(\mI' + \mJ)}^{R}(t) - P_{R/(\mI + \mJ')}^{R}(t) &= t P_{ \mI \mJ }^{R}(t) + t^2 P_{ \mI \mJ }^{R}(t) - P_{R/(\mI + \mJ)}^{R}(t).
\end{align*}
Isolating $t(1+t)$ on the right-hand side provides equation \ref{eqn:pon1}.

To get equation \ref{eqn:pon2}, recall that $\mI$ and $\mJ$ are generated by a regular sequence and thus $\mX$ and $\mY$ are Koszul complexes. It follows that $P_{R/\mI}^{R}(t) = (1 + t)^m$ and $P_{R/\mJ}^{R}(t) = (1 + t)^n$. These expessions and equation \ref{align:denom} combine with equation \ref{eqn:pon1} to produce equation \ref{eqn:pon2}.\end{proof}

Lastly, we explicitly state how to combine Example \ref{Thm:TorInd} with Theorem \ref{thm:minimal} to obtain Theorem \ref{thm:main}.

\begin{cor}\label{cor:fiber}
Suppose the rings $R$, $A$, and $B$ satisfy either case of Example \ref{Thm:TorInd}. If $I' \subseteq \maxM_A^2$ and $J' \subseteq \maxM_B^2$, then $\Cone{\Omega}$ is a minimal free resolution over $R$ of $\frac{A}{I'} \times_k \frac{B}{J'} \cong \frac{R}{\langle \mI', \vv{x}\vv{y}, \mJ' \rangle}$ where $\mI' = I' \otimes_A R$ and $\mJ' = J' \otimes_B R$.
\end{cor}

\section*{Acknowledgements}

A special thank you to my Ph.D. advisor, Dr. Sean Sather-Wagstaff, for their guidance while working on this paper. I also want to thank Frank Moore, Rachel Diethorn, and Keller Vandebogert for pointing me in the direction of various helpful sources.

\bibliographystyle{plain}

\begin{thebibliography}{10}

\bibitem{celikbas2020weakly}
Ela Celikbas, Olgur Celikbas, C\u{a}t\u{a}lin Ciuperc\u{a}, Naoki Endo, Shiro
  Goto, Ryotaro Isobe, and Naoyuki Matsuoka.
\newblock Weakly {A}rf rings, 2020.

\bibitem{MR2580452}
Lars~Winther Christensen, Janet Striuli, and Oana Veliche.
\newblock Growth in the minimal injective resolution of a local ring.
\newblock {\em J. Lond. Math. Soc. (2)}, 81(1):24--44, 2010.

\bibitem{MR1322960}
David Eisenbud.
\newblock {\em Commutative algebra}, volume 150 of {\em Graduate Texts in
  Mathematics}.
\newblock Springer-Verlag, New York, 1995.
\newblock With a view toward algebraic geometry.

\bibitem{freitas2019vanishing}
Thiago~H. Freitas, Victor Hugo~Jorge Pérez, Roger Wiegand, and Sylvia Wiegand.
\newblock Vanishing of tor over fiber products, 2019.

\bibitem{MR647683}
Jack Lescot.
\newblock La s\'{e}rie de {B}ass d'un produit fibr\'{e} d'anneaux locaux.
\newblock {\em C. R. Acad. Sci. Paris S\'{e}r. I Math.}, 293(12):569--571,
  1981.

\bibitem{MR1011461}
Hideyuki Matsumura.
\newblock {\em Commutative ring theory}, volume~8 of {\em Cambridge Studies in
  Advanced Mathematics}.
\newblock Cambridge University Press, Cambridge, second edition, 1989.
\newblock Translated from the Japanese by M. Reid.

\bibitem{MR2488551}
W.~Frank Moore.
\newblock Cohomology over fiber products of local rings.
\newblock {\em J. Algebra}, 321(3):758--773, 2009.

\bibitem{MR2711725}
William~F. Moore.
\newblock {\em Cohomology of products of local rings}.
\newblock ProQuest LLC, Ann Arbor, MI, 2008.
\newblock Thesis (Ph.D.)--The University of Nebraska - Lincoln.

\bibitem{MR3691985}
Saeed Nasseh and Sean Sather-Wagstaff.
\newblock Vanishing of {E}xt and {T}or over fiber products.
\newblock {\em Proc. Amer. Math. Soc.}, 145(11):4661--4674, 2017.
\newblock [Paging previously given as 1--14].

\bibitem{MR3862678}
Saeed Nasseh, Sean Sather-Wagstaff, Ryo Takahashi, and Keller VandeBogert.
\newblock Applications and homological properties of local rings with
  decomposable maximal ideals.
\newblock {\em J. Pure Appl. Algebra}, 223(3):1272--1287, 2019.

\bibitem{nasseh2018gorenstein}
Saeed Nasseh, Ryo Takahashi, and Keller VandeBogert.
\newblock On gorenstein fiber products and applications, 2018.

\bibitem{MR3988200}
Hop~D. Nguyen and Thanh Vu.
\newblock Homological invariants of powers of fiber products.
\newblock {\em Acta Math. Vietnam.}, 44(3):617--638, 2019.

\bibitem{MR3754407}
Hamidreza Rahmati, Janet Striuli, and Zheng Yang.
\newblock Poincar\'{e} series of fiber products and weak complete intersection
  ideals.
\newblock {\em J. Algebra}, 498:129--152, 2018.

\bibitem{vandebogert2020vanishing}
Keller VandeBogert.
\newblock Vanishing of avramov obstructions for products of sequentially
  transverse ideals, 2020.

\bibitem{MR2262383}
Daniel Visscher.
\newblock Minimal free resolutions of complete bipartite graph ideals.
\newblock {\em Comm. Algebra}, 34(10):3761--3766, 2006.

\end{thebibliography}

\end{document}